\documentclass[12pt,reqno]{amsart}

\usepackage{amssymb,amsmath,amsthm,mathtools,wasysym,calc,verbatim,enumitem,tikz,pgfplots,hyperref,url,mathrsfs,cite,fullpage,bbm,comment}
\usepackage{comment}
\mathtoolsset{showonlyrefs}
\pgfplotsset{compat=1.18}
\newtheorem{theorem}{Theorem}

\newtheorem{lemma}[theorem]{Lemma}
\newtheorem{claim}[theorem]{Claim}

\newtheorem{observation}[theorem]{Observation}

\newtheorem*{claim*}{Claim}

\theoremstyle{remark}
\newtheorem*{remark*}{Remark}

\numberwithin{theorem}{section}

\renewcommand{\phi}{\varphi}

\renewcommand{\leq}{\le}
\renewcommand{\geq}{\ge}

\newcommand{\eps}{\varepsilon}

\newcommand{\cA}{\mathcal A}

\newcommand{\cC}{\mathcal C}

\newcommand{\cO}{\mathcal O}

\newcommand{\cI}{\mathcal I}

\def\1{\mathbbm{1}}

\renewcommand{\le}{\leqslant}
\renewcommand{\ge}{\geqslant}
\renewcommand{\P}{\mathbb{P}}

\newcommand{\PP}{\mathbb P}

\newcommand{\Span}{\mathrm{span}}

\title{A polynomial improvement for the odd cycle--complete Ramsey numbers} 
\author{Marcelo Campos, Matthew Jenssen, Marcus Michelen, Florian Pfender, \\ Julian Sahasrabudhe}

\address{Instituto de Matem\'{a}tica Pura e Aplicada (IMPA).}
\email{marcelo.campos@impa.br}
\address{King's College London, Department of Mathematics.}
\email{matthew.jenssen@kcl.ac.uk}
\address{Northwestern University. Department of Mathematics.}
\email{michelen.math@gmail.com}
\address{University of Colorado Denver. Department of Mathematical and Statistical Sciences.}
\email{florian.pfender@ucdenver.edu}
\address{University of Cambridge. Department of Pure Mathematics and Mathematical Statistics.} 
\email{jdrs2@cam.ac.uk}

\begin{document}

\thanks{Marcelo Campos is supported by Serrapilheira grant R-2412-51283. Matthew Jenssen is supported by a UK Research and Innovation Future Leaders Fellowship MR/W007320/2. Marcus Michelen is supported in part by NSF grants DMS-2336788 and DMS-2246624. Florian Pfender is supported in part by NSF FRG grant DMS-2152498. Julian Sahasrabudhe is supported by European Research
Council (ERC) Starting Grant “High Dimensional Probability and Combinatorics”, grant
No. 101165900}

\begin{abstract}
We give a polynomial improvement to the cycle-complete Ramsey numbers 
\[ r(C_{\ell},K_k) \geq k^{1+\frac{1}{\ell-2} + \eps_{\ell} + o(1)}, \]
for all fixed odd $\ell > 7$ with $k \rightarrow \infty$, for some $\eps_{\ell} > 0$.
\end{abstract}

\maketitle

\vspace{-2em}

\section{Introduction}

The \emph{cycle-complete} Ramsey number $r(C_{\ell},K_k)$ is the smallest $n$ so that every graph on $n$ vertices has either a cycle of length $\ell$ or an independent set of size $k$. Here our interest lies in the case when $\ell$ is fixed and odd and $k$ tends to infinity. The study of these numbers goes back to the foundational work of Erd\H{o}s \cite{erdos1959graph} who proved 
\begin{equation}\label{eq:erdos-bounds} c_{\ell} k^{1+\frac{1}{2\ell}} \leq r(C_{\ell},K_k) \leq c_{\ell}' k^{1+\frac{2}{\ell-1}} ,\end{equation}
where the upper bound holds for all odd $\ell > 1$ (he proved a similar result when $\ell$ is even) and the lower bound was proved using an early application of the probabilistic method. The lower bound was then improved by Spencer \cite{spencer1977asymptotic} in 1977 who used the Lov\'asz local lemma  to show
\begin{equation}\label{eq:spencer} r(C_{\ell},K_k) \geq k^{1+\frac{1}{\ell-2} +o(1) }.\end{equation}
 While this bound has been improved by logarithmic factors by Bohman and Keevash \cite{bohman2010early} 
and by Verstra\"{e}te and Mubayi \cite{mubayi2024}, the exponent in \eqref{eq:spencer} has stood as the best known bound on these numbers, apart from the special cases $\ell \in \{5,7\}$, which we discuss below. In this paper we give the first polynomial improvement to this bound for all odd $\ell>7$. 

\begin{theorem}\label{thm:main1} For all odd $\ell > 7$, we have that 
\[ r(C_{\ell},K_k) \geq k^{1+\frac{1}{\ell-2} + \eps_{\ell} + o(1)},\]
where $\eps_{\ell} = \Theta(\ell^{-2})$.
\end{theorem}
In fact we can take $\eps_{\ell} = ((\ell-2)(2\ell-5))^{-1}$ in the statement of our main theorem.

What is perhaps most interesting about this result is that the bound \eqref{eq:spencer} represents a natural barrier for both this and many other problems in combinatorics. Indeed, a simple way to construct a $C_\ell$-free graph with small independence number is to  sample $G \sim G(n,p)$ and then delete an edge from each instance of $C_{\ell}$.  Here one needs to choose $p$ so that 
\[ p^{\ell}n^{\ell} \ll pn^2, \]
to ensure that only a small fraction of the edges of $G$ are removed. This is known as \emph{the deletion threshold} for the cycle $C_{\ell}$ and the bound at \eqref{eq:spencer} reflects the independence number of $G(n,p)$ precisely at this deletion threshold. In many cases in combinatorics we know little about how to obtain good examples that are significantly beyond this threshold.

In this paper we show that we can push well beyond this deletion threshold for the cycle complete Ramsey numbers, using a new idea of randomly superimposing blow-ups of random graphs. This idea is due to Hefty, Horn, King and Pfender \cite{hefty2025}, which was inspired by a previous idea of Campos, Jenssen, Michelen and Sahasrabudhe \cite{campos2025}.  In \cite{hefty2025}, this yields the lower bound on the extreme off diagonal Ramsey numbers $R(3,k)$
\[
R(3,k)\ge \bigg( \frac{1}{2} +o(1) \bigg) \frac{k^2}{\log k}
,\] which is within a factor of $2+o(1)$ of the best upper bound due to Shearer \cite{shearer1983note}, and is conjectured to be sharp \cite{campos2025}.

It seems reasonable to guess, for odd $\ell$, that Erd\H{o}s's bound \eqref{eq:erdos-bounds} gives the correct exponent, 
\[ r(C_{\ell},K_k) = k^{1+\frac{2}{\ell-1}+o(1)},\]
which has been tentatively conjectured and supported by some recent conditional results~\cite{conlon2024ramsey}. It appears that the construction in this paper cannot be directly used to obtain an improvement to the lower bound beyond a factor of $n^{\eps_{\ell} + o(1)}$ where $\eps_{\ell} = \Theta(\ell^{-2})$.

\subsection{A brief history of the problem}
Before we turn to give the short proof of our main theorem, we briefly elaborate on the history of the problem. While the upper and lower bounds stated above are the best known for general odd $\ell$, up to lower order factors, there have been several interesting advances on the lower order terms. The best known upper bound is due to Sudakov \cite{sudakov2002note} and, independently, Li and Zang \cite{li2003independence} who proved
\[ r(C_{\ell}, K_k) \leq c_{\ell} k^{1+\frac{2}{\ell-1}} (\log k)^{-\frac{2}{\ell-1}}\] for all odd $\ell >2$. The lower bounds were also improved by poly-logarithmic factors, first in the influential work of Bohman and Keevash~\cite{bohman2010early} on the $H$-free process and then, more recently, by Mubayi and Verstra\"{e}te \cite{mubayi2024} via a new method which established an important connection between lower bounds for the Ramsey numbers and spectral expanders.

There have also been two recent \emph{polynomial} improvements to the cycle-complete Ramsey numbers when $\ell \in \{ 5,7 \}$. Mubayi and Verstra\"{e}te \cite{mubayi2024} were the first to break the deletion threshold barrier for $C_5$ and $C_7$ by showing 
 \begin{equation*}
    r(C_5,K_k) \geq (1 + o(1))k^{11/8} \qquad \text{ and } \qquad r(C_7,K_k) \geq (1 + o(1))k^{11/9}\,,
\end{equation*}
using an ingenious construction based on incidence
graphs of generalized hexagons. These bounds were then further improved by Conlon, Mattheus, Mubayi and Verstra{\"e}te~\cite{conlon2024ramsey} who showed that
\begin{equation}\label{eq:CMMV}
    r(C_5,K_k) \geq k^{10/7 -o(1)} \qquad \text{ and } \qquad r(C_7,K_k) \geq k^{5/4-o(1)}\, ,
\end{equation}
by leveraging ideas from the recent breakthrough of Mattheus and Verstra\"ete on the Ramsey numbers $R(4,k)$\cite{mattheus2024asymptotics}.

In the case of even $\ell$, there are no known examples that beat the deletion threshold by a polynomial factor. Again \cite{bohman2010early} and \cite{mubayi2024} supply us with poly-logarithmic improvements to the lower bound. For the upper bound, Erd\H{o}s proved in the 1950s that  
 \[  r(C_\ell,K_k) \leq c_{\ell} n^{1 + \frac{2}{\ell-2}}\,, \]
which has only been improved by polylogarithmic factors by Caro, Li, Rousseau, and Zhang \cite{caro2000asymptotic}. For even $\ell$ it is much less clear if Erd\H{o}s's upper bound is sharp. Indeed, it is a beautiful conjecture of Erd\H{o}s that 
\[ r(C_4,K_k) \leq k^{2-\eps + o(1)},\]
for some $\eps>0$.

Finally we mention that when $k$ is fixed and $\ell \rightarrow \infty$ the Ramsey numbers $r(C_{\ell},K_k)$ have also received quite a bit of attention. As this regime enjoys a somewhat different behavior from what we see for fixed $\ell$, we defer to the paper of Keevash, Long and Skokan \cite{keevash2021cycle}.

\section{The Construction} 
Let $\ell \geq 5$ be an odd integer. Define $p_c$ to be the \emph{deletion threshold} for $C_{\ell}$ in $G(n,p)$: the value for which $p_c^{\ell}n^{\ell} = p_c n^2$. That is,
\[  p_c = n^{-1+\frac{1}{\ell-1}}.\]
We now define our parameters 
\begin{equation}\label{eq:p-def}  \eps = ((2\ell-3)(\ell-1))^{-1}, \qquad p = p_c n^{\eps}(\log n)^{-1} , \qquad \text{ and } \qquad r =  (pn)^{1/2}(\log n)^{-3/2} .
\end{equation}
We note that this $\eps$ is not the same as the $\eps_\ell$ that appears in the statement of Theorem~\ref{thm:main1}. We construct a $C_\ell$-free graph 
on $(1-o(1))n$ vertices with density $\approx 2p$ with no independent sets of size $k$, where
\begin{equation}  k =8p^{-1}( \log n)^3 . \end{equation}
We will see that Theorem~\ref{thm:main1} then follows immediately. For future reference, we note that our parameters satisfy \footnote{We write $f(n)\ll g(n)$ to denote that $f(n)/g(n)\to 0$ as $n\to\infty$.}

\begin{equation}\label{eq:inequalities} p^{\ell-1}n^{\ell-2}/r \ll (\log n)^{-2},\qquad r(pn)^{\ell-2} \ll n\qquad \text{and} \qquad  r \leq 4\sqrt{n/k}. \end{equation}

Throughout this paper, we assume that 
$n$ is sufficiently large for all stated inequalities to hold.

We now describe the construction in the following three short subsections.

\subsection{Superimposing blowups of random graphs}
We let  
\begin{align}\label{eq:randomparts}
[n] = V_1 \cup \cdots \cup V_{n/r} = W_1 \cup \cdots \cup W_{n/r} 
\end{align}
be two independent random partitions of $[n]$ into sets of size $r$. We call these the \emph{red} and \emph{blue} partitions of $[n]$ respectively (not to be confused with the red and blue colours of a Ramsey colouring). We now define $G$ to be a random graph defined in the following steps. We first independently sample two graphs 
\[ G^{\circ}_{R},\, G^{\circ}_{B} \sim G(n/r,p) . \]  We then define 
 $G'_R$ to be the $r$-blow up of $G^{\circ}_{R}$, with parts $V_1,\ldots,V_{n/r}$ and define $G'_B$ to be the $r$-blow up of $G^{\circ}_{B}$, with parts $W_1,\ldots,W_{n/r}$. We then define the multigraph
 \[ G' = G_R' \cup G_B',\]
 keeping both edges whenever an edge is in both $G_R'$ and $G_B'$. Throughout, we will refer to $G'_R$ as the red graph and $G'_B$ as the blue graph.

 We now construct the final graph $G$ from $G'$ in two steps. In a first step we delete vertices to deal with certain ``degenerate'' cycles in $G'$. Then in a second step, we delete edges to deal with the majority of the cycles.

\subsection{Vertex deletion step}\label{sec:vertex-deletion}
For this we now introduce a bit of notation and terminology. For $S \subset [n]$, define the ``projections''
\begin{equation}\label{eq:partition-pseuo-random} 
\pi_R(S) = \big\{ i : S \cap V_i \not= \emptyset \big\} \qquad \text{ and } \qquad \pi_B(S) = \big\{ i : S \cap W_i \not= \emptyset \big\} .
\end{equation} We say that two edges $e,f$ are \emph{parallel} if $\pi_R(e) = \pi_R(f)$ or $\pi_B(e) = \pi_B(f)$. Say that $H \subset G'$ is \emph{degenerate} if it contains parallel edges. In our vertex deletion step we will remove all degenerate $\ell$-cycles. We will do this by removing a slightly more general type of subgraph which we now describe. 

For $x_1,\ldots,x_t\in [n]$, we say that $B=x_1\ldots x_t$ is a \emph{broken path} of order $t$, if $B$ is a disjoint union of paths in $G'$; 
\[ P_1=x_{1}\ldots x_{i_1}, ~P_2=x_{i_1+1}\ldots x_{i_2},\quad \ldots, \quad ~P_s=x_{i_{s-1}+1}\ldots x_{t},\] where $1\le s\le t$, and $\pi_R(x_{i_j})=\pi_R( x_{i_j+1})$ or $\pi_B(x_{i_j})=\pi_B( x_{i_j+1})$ for each $1\le j\le s-1$. Here paths may have length $0$. A \emph{broken cycle} is a broken path $x_1\ldots x_t$ with the additional edge $x_tx_1$.

We now define $D$, the vertices of $G'$ which we will remove, to be the union of the vertices of all non-degenerate 
broken cycles $B$ in $G'$ of order $t\leq \ell-1$ which contain an odd number of edges. 

\subsection{Edge deletion step}\label{sec:edge-deletion} 
We now describe a set of edges $F$ which we will remove from $G'$ to form $G$. To prepare the ground for this, we define a \emph{red ordering} on $[n]^{(2)}$ which is arbitrary apart from the fact that all pairs $V_i \times V_j$ are ordered consecutively. Similarly, we define a \emph{blue ordering} on $[n]^{(2)}$ to be arbitrary apart from the pairs $W_i\times W_j$ are ordered consecutively.

Now for every non-degenerate $\ell$-cycle $C \subset G'$, we identify an edge $e_C \in G'$ for deletion. For such $C \subset G'$, the edges of this cycle are coloured red and blue by definition. If the cycle is monochromatic red, we define $e_C$ to be the largest $e \in C$ in the red ordering. Likewise if the cycle is monochromatic blue, we define it to be the largest edge in the blue ordering. If both colours appear on $C$, say red is the minority colour. In this case, we define $e_C$ to be the the largest red $e \in C$, in the red ordering. The blue case is defined symmetrically. 

We then define $F$, the edges marked for deletion, as  
\[ F = \{ e_C : C \subset G' \text{ an } \ell\text{-cycle} \}.\]

Finally, we define $G$ to be $G'$ with a) all of the edges $F$ removed; b) all of the vertices $D$ deleted; and c) any multiple edges removed so that the remaining graph is simple. This concludes the definition of $G$.

Our main technical result of this paper is the following.

\begin{theorem}\label{thm:main2} For $\ell \geq 5$ and odd, $G$ is a $C_{\ell}$-free graph with $|V(G)| = (1-o(1))n$ and 
\[ \alpha(G) \leq 8p_\ell^{-1} (\log n)^{3},\]
with high probability, where $p = p_{\ell}$ is defined at \eqref{eq:p-def}. 
\end{theorem}

Note that Theorem~\ref{thm:main2} implies Theorem~\ref{thm:main1}. Indeed if we set
\[
    k=8p^{-1} (\log n)^{3} = n^{1-\frac{1}{\ell-1}-\frac{1}{(\ell-1)(2\ell-3)}+o(1)},
\]
Theorem~\ref{thm:main2} says there exists a $C_{\ell}$-free graph on $(1-o(1))n$ vertices with no independent set of size $k$. Thus
\begin{align*}
    r(C_\ell,K_k)\ge (1-o(1))n=k^{1+\frac1{\ell-2}+\frac1{(\ell-2)(2\ell-5)}+o(1)}.
\end{align*}

\subsection{Heuristic discussion of construction}
 We seek to construct a graph that is $C_{\ell}$-free 
 and is sufficiently dense and ``pseudo-random'' to meet all $k$-sets in an edge. Today, the standard method of achieving these ends is via the ``deletion method'', mentioned above. Here one first samples a random graph $G' \sim G(n,p)$ and then deletes an edge from each cycle of length $\ell$ to form the graph $G$. For this method to have any chance of working, one needs that the number of copies of $C_{\ell}$ in this random sample is significantly smaller than the number of edges; otherwise, too many edges will be deleted in the deletion step. In the context of $C_{\ell}$-free graphs this means 
 \[ p^{\ell}n^{\ell} \ll pn^2 .\] 

The key behind our construction is that it improves this trade-off in the deletion step, while remaining pseudo-random. In the standard deletion method proof, the deletion of \emph{one} edge results in the removal of \emph{one} copy of $C_{\ell}$. For us, the deletion of one edge will result in the removal of $r = n^{\eps}$ copies of $C_{\ell}$. Thus we shall need the density only to satisfy 
\[ p^{\ell}n^{\ell} \ll r \cdot pn^2, \]
which allows us to take the density a polynomial factor denser than in the standard deletion method.  

To see where this $1$ to $r$ ratio ``comes from'', we make the following simple observation, on which the whole paper hinges. Recall that $G'$ is our graph $G$ before the deletion step.

\begin{observation} Let $C \subset G'$ be an $\ell$-cycle. There exists $e \in C$ which is in at least $r - \ell$ $\ell$-cycles.
\end{observation}
\begin{proof} $C$ is two coloured by the red and blue graphs $G'_R,G_B'$. Since $\ell$ is odd there exists a vertex $x \in V(C)$ for which both incident edges are the same colour. Without loss, assume this colour is red. Now fix some $e \in C$ not incident with $x$. Then each of the sets 
\[ V(C) \setminus \{ x \} \cup \{ x'\}  \qquad \text{ where } \qquad \pi_R(x') = \pi_R(x) \hspace{1mm} \text{ and }x'\notin V(C)\] 
are cycles which include $e$. 
\end{proof}

Thus we see that the deletion of $e$, from the above observation, removes $r-O(1)$ cycles. 

In practice, we don't just care that the number of edges removed is small but also that the removed edges are ``well-spread'' throughout the graph. Indeed the main technical obstacle in this paper is to show that 
the graph after the deletion step remains sufficiently pseudo-random so that the independence number of $G$ mimics that of the random graph of the same density. In our context, we prove
\[ \alpha(G) \leq p^{-1}(\log n)^{\Theta(1)}.\]
Formally, this observation is relevant in Lemma~\ref{lem:closed-edges}, when we are calculating the number of closed pairs in a $k$-set. 

\section{Removing degenerate cycles}
In this section we prove that $D$ intersects all degenerate copies of $C_{\ell}$ in $G'$. 

\begin{lemma}
\label{lem:vx-del}
$G'-D$ contains only non-degenerate copies of $C_\ell$.
\end{lemma}

We also need to show that $D$ is relatively small. 

\begin{lemma}\label{lem:vx-del-2}  With high probability, $|D|=o(n)$.
\end{lemma}

Before turning to the proofs of these lemmas, we record the following simple observation. 

\begin{observation}\label{obs:broken-cycle}
Let $B$ be a broken cycle with an odd number of edges. Then there exists a non-degenerate broken cycle $B'$ with an odd number of edges with $V(B') \subset V(B)$.
\end{observation}
\begin{proof}
We prove the claim by applying induction on the number of edges in $B$. If $B$ has one edge then $B$ is non-degenerate and we are done. So assume $B$ contains at least three edges. 
Now, if $B$ is degenerate then there exists a pair of parallel edges $e = x_1y_1$, $f = x_2y_2$. By symmetry we may assume that $\pi_R(x_1)=\pi_R(x_2)$ and $\pi_R(y_1)=\pi_R(y_2)$. 

Now, let $P$ denote the broken path from $x_1$ to $x_2$ in $B$ that does not contain $y_1$ and let $Q$ denote the broken path from $y_1$ to $y_2$ in $B$ that does not contain $x_1$. (Note that one of $P$, $Q$ may be trivial).

\usetikzlibrary{calc, positioning}

\tikzset{vtx/.style={circle, draw,fill=white, inner sep=1.5pt}}

\begin{figure}[h!]
\begin{center}
    \begin{tikzpicture}
    \draw
    (-0.1,0) node[vtx](x1){$x_1$} 
    (1.1,0) node[vtx](y1){$y_1$} 
    (-0.1,2) node[vtx](x2){$x_2$} 
    (1.1,2) node[vtx](y2){$y_2$} 
    (-1,1) node {$P$}
    (2,1) node {$Q$}
    (0.5,0.3) node {$e$}
    (0.5,1.6) node {$f$}
    ;
    \draw
    (x1) --(y1)
    (x2) -- (y2);
    \draw[line width=0.5mm, rounded corners = 8mm]
    (x1) -- (-1.5,0) -- (-1.5,2) -- (x2);
    \draw[line width=0.5mm, rounded corners = 8mm]
    (y1) -- (2.5,0) -- (2.5,2) -- (y2);
    \end{tikzpicture}
    \hspace{1.5cm}
    \begin{tikzpicture}
    \draw
    (-0.1,0) node[vtx](x1){$x_1$} 
    (1.1,0) node[vtx](y1){$y_1$} 
    (-0.1,2) node[vtx](y2){$y_2$} 
    (1.1,2) node[vtx](x2){$x_2$} 
    (-1,1) node {$P$}
    (2,1) node {$Q$}
    (0.5,0.3) node {$e$}
    (0.5,1.6) node {$f$}
    (1.2,1.95) node (x2-){}
    (-0.2,2.05) node (y2+){}
    ;
    \draw
    (x1) --(y1);
    \draw[line width=0.5mm, rounded corners = 8mm]
    (x1) -- (-1.5,0) -- (-1.5,1.95) -- (x2-);
    \draw[line width=0.5mm, rounded corners = 8mm]
    (y1) -- (2.5,0) -- (2.5,2.05) -- (y2+);
    \draw
     (-0.1,2) node[vtx](y2){$y_2$} 
    (1.1,2) node[vtx](x2){$x_2$} 
    ;
    \end{tikzpicture}
\end{center}
\caption{Possible broken paths in the proof of Observation~\ref{obs:broken-cycle}}
\label{fig:PQ}
\end{figure}

There are now two cases (depicted in Figure~\ref{fig:PQ}). In the first case, $f$ is in neither of $P,Q$. In this case, $E(B) = E(P) \cup E(Q) \cup \{e\} \cup \{f\}$ is a partition and thus we see that both $|E(P)|,|E(Q)| < |E(B)|$. Moreover, one $E(P),E(Q)$ must be odd.

In the other case, $f$ is in \emph{both} of $P,Q$. Then $E(B)\setminus \{f\} = (E(P)\setminus \{f\}) \cup (E(Q)\setminus\{f\}) \cup \{e\}$ is a partition. From this we deduce, again, we see that $|E(P)|,|E(Q)| < |E(B)|$ and that one must be odd. 

In either case, without loss, assume that $P$ is our desired path; that is $|E(P)|<|E(B)|$ and odd. As $\pi_R(x_1)=\pi_R(x_2)$, $P$ is a broken cycle. We thus apply induction to $P$. 
\end{proof}

We now turn to prove Lemma~\ref{lem:vx-del}. Recall that an odd cycle is also a broken cycle with an odd number of edges.

\begin{proof}[Proof of Lemma~\ref{lem:vx-del}]
Assume $C \subset G'$ is a degenerate $C_{\ell}$. By Observation~\ref{obs:broken-cycle} there is a non-degenerate broken cycle $B$ with $V(B) \subset V(C)$ with an odd number of edges. By definition, we have $V(B)\subset D$, and thus $D\cap V(C)\ne \emptyset$. So $C \not\subset G-D$, as desired.
 \end{proof}

We now prove Lemma~\ref{lem:vx-del-2}. For this, we note that if $H \subset [n]^{(2)}$ is non-degenerate, then
\begin{equation}\label{eq:simp-prob} \PP( H \subset G') \leq (2p)^{e(H)}.\end{equation}

\begin{proof}[Proof of Lemma~\ref{lem:vx-del-2}]  
We show the expected number of non-degenerate broken cycles with an odd number of edges is at most $o(n)$. We then apply Markov's inequality to finish. We claim that the expected number of non-degenerate broken cycles in $G'$ of order $2\leq t\leq \ell-1 $ with $1\le a\le t-1$  edges is at most
\[
2^\ell (2r)^{t-a} n^{a} (2p)^a\, 
.\]
Indeed, if we write the broken cycle $x_1\ldots x_{t}$ with $x_tx_1$ being an edge, there are $\leq 2^{t-1}$ ways of specifying which pairs $x_ix_{i+1}$ ($i\neq t$) are edges of the broken cycle. There are then $n$ choices for $x_1$ and $\leq 2r$ choices for $x_2$ if $x_1x_2$ is not an edge and $\leq n$ choices for $x_2$ otherwise. It follows inductively that there are $\leq 2^{t-1} (2r)^{t-a} n^{a}$ choices for $x_1,\ldots, x_t$. Since the broken cycle is non-degenerate, the probability that the $a$  edges appear in $G'$ is at most $(2p)^a$ from \eqref{eq:simp-prob}. Recalling~\eqref{eq:inequalities}, we have
\[
r^{t-a}n^a p^a \leq r (np)^{\ell-2} =o(n), \] as desired.\end{proof}

\section{Basic Pseudo-random properties}
In this section we define a pseudo-randomness property $\cA$ for the graphs $G^{\circ}_R, G^{\circ}_B, G'$ and show that $\cA$ holds with high probability. The first property we include into $\cA$ is basic control on the degrees. For all $x \in V(G_R^{\circ})$ we have 
\begin{equation}\label{eq:degrees} d_{G^{\circ}_R}(x) = (1+O(n^{-c}))p n/r, \end{equation}
where $c=\tfrac{1}{4(\ell-1)}$ and likewise for $G_B^{\circ}$. The second property we include into $\cA$ is the following spectral pseudo-randomness property for the pre-blown up graphs $G^{\circ}_R$ and $G^{\circ}_B$. We let $A_R,A_B$ be the adjacency matrices of $G_R^{\circ}, G_B^{\circ}$ respectively. We include the event 
\begin{equation}  \label{eq:spectrum}  \|A_R - pJ_{n/r}\|_{op} \leq 3\sqrt{pn/r}  \end{equation}
and likewise for $A_B$. Here $J_{t}$ denotes the $t$ by $t$ matrix of all ones and $\|\cdot\|_{op}$ denotes the operator norm.

We also need a pseudo-randomness property that captures the typical interaction of the red and blue partitions $\{V_i\}_i$ and $\{ W_i \}_i$. We include into the event $\cA$ the following. For every $t \leq k$ and $S \in [n]^{(t)}$ we have 
\begin{equation} \label{item:A2} \max\big\{ |\pi_R(S)|,  |\pi_B(S)|\big\} \geq \delta t\, , \end{equation} where we set 
\begin{equation}\label{eq:def-delta} \delta  = (\log n)^{-1},\end{equation} for use throughout the paper. On a technical note, we require that there are not too many edges incident to a vertex $v$ that are in both $G_R'$ and $G_B'$. Here we include into $\cA$ the weak bound \begin{equation}\label{eq:double-degree} 
    \max_{v \in [n]} \big| N_{G_R'}(v) \cap N_{G_B'}(v) \big| \leq n^{1-\eta} p, 
\end{equation}
where $\eta=1/(4\ell)$.
Finally, we need the graph $G'$ to have good vertex expansion properties. For a graph $H$ and set of vertices $S \subset V(H)$ define 
\[\partial_H(S) = \big\{v \in V(H) \setminus S : N_H(v) \cap S \neq \emptyset\big\}\] to be the vertex boundary of $S$.  We include into the event $\cA$ the property that every set $S \subset [n]$ with $|S| \leq (2p)^{-1}$ satisfies \begin{equation}\label{eq:vertex-boundary}
|\partial_{G'}(S)|\geq \delta pn|S|/8\, .
\end{equation}
We now turn to show that properties \eqref{eq:degrees},\eqref{eq:spectrum},\eqref{item:A2},\eqref{eq:double-degree} and \eqref{eq:vertex-boundary} hold with high probability. The only elements that are not standard, or follow from well-known results are \eqref{item:A2} and \eqref{eq:vertex-boundary}. We isolate these here. 

\begin{observation}\label{obs:quasi-random-partition}
The event \eqref{item:A2} holds with high probability.    
\end{observation}
\begin{proof}
For $t \leq k$, fix $S\in [n]^{(t)}$ and consider the probability that $|\pi_R(S)|<\delta t$. By symmetry, this is the same as fixing the partition and choosing $S\in [n]^{(t)}$ uniformly at random. Thus
\[ \PP\big(|\pi_R(S)|<\delta t \big) \leq t\binom{n/r}{\delta t}\binom{r\delta t}{t} \binom{n}{t}^{-1}\, ,\]
since there are $n/r$ parts of which we choose $<\delta t$ for the red projection. We then choose $S$ from the union of these parts, which has size $< r(\delta t)$.

By independence of the red and blue partitions we have 
\[ 
\PP\big(|\pi_R(S)|<\delta t\, \wedge \, |\pi_B(S)|<\delta t\big) \leq \bigg[t\binom{n/r}{\delta t}\binom{r\delta t}{t} \binom{n}{t}^{-1} \bigg]^2\, 
\leq \bigg( \frac{C\delta^2r^2}{\big(n/t\big)}\bigg)^t \binom{n}{t}^{-1}\, ,\]
for some constant $C>0$.
We note that, by our choices of $r,k$ at \eqref{eq:p-def}, we have $r\leq  4(n/k)^{1/2}$. Thus the above is $\ll 2^{-t}\binom{n}{t}^{-1}$. Thus we may union bound over all $S$ and $t$ to obtain the desired result.
\end{proof}

\begin{observation}\label{obs:vxexpand} 
The event \eqref{eq:vertex-boundary} holds with high probability. 
\end{observation}
\begin{proof}
By Observation~\ref{obs:quasi-random-partition} it suffices to show that \eqref{eq:vertex-boundary} holds with high probability conditioned on~\eqref{item:A2} holding.
Let $S \subset [n]$ satisfy $|S|\leq (2p)^{-1}$ and let $S_0 \subset S$ be a subset with the property $|\pi_R(S_0)|=|S_0|$ so that $S_0$ contains exactly one element from each $V_i$ that $S$ intersects. Write $T = \pi_R(S_0) \subset[n/r]$. Without loss, by \eqref{item:A2}, we have $|T| = |S_0|\geq \delta |S|$. Now write $t = |T|$ and note 
\[ \big| \partial_{G_R^\circ}(T) \big|  \sim  \text{Bin}\big(n/r-t, 1-(1-p)^{t}\big), \] which has mean $\geq tpn/(4r)$ since $pt\leq 1/2$. By Chernoff and a union bound, we have
\[
|\partial_{G^\circ_R}(T)|\geq  tpn/(8r)
\]
for all $T\subset[n/r]$ with $t\leq (2p)^{-1}$, with high probability. Thus, with high probability,
\begin{equation*}
|\partial_{G'}(S)|\geq r |\partial_{G^\circ_R}(T)|\geq |S_0|pn/8 \geq  \delta |S|pn/8 \, ,
\end{equation*}
where for the first inequality we observe that if $i\in \partial_{G^\circ_R}(T)$ then $V_i\subset \partial_{G'}(S)$.
\end{proof}

We now prove that $\cA$ holds with high probability.

\begin{lemma} $\cA$ holds with high probability. 
\end{lemma}
\begin{proof} 
By the Chernoff bound and a union bound over $x\in V(G_R^\circ)$ we have 
\[ 
d_{G^{\circ}_R}(x) = pn/r + \sqrt{pn/r} \cdot \log n \quad \text{for all} \quad x\in V(G_R^\circ)
\]
with high probability,
 and likewise for $G_B^{\circ}$. Recalling the definition of $p$ and $r$ from~\eqref{eq:p-def} then shows that \eqref{eq:degrees} holds with high probability.
The fact that \eqref{eq:spectrum} holds follows from Theorem 1.4 in \cite{vu2005}. We just saw that \eqref{item:A2} holds by Observation~\ref{obs:quasi-random-partition} and \eqref{eq:vertex-boundary} holds by Observation~\ref{obs:vxexpand}. 

To prove \eqref{eq:double-degree}, we may assume \eqref{eq:degrees} holds in which case we have $|N_{G_R'}(v)| = (1 + o(1)) np$ for all $v$. Fix $S\subset [n]$ and note that
\[
|N_{G_B'}(v)\cap S| = \sum_{i=1}^{n/r} |W_i \cap S|\cdot \xi_i
\]
where the $\xi_i$ are i.i.d.\ $\text{Ber}(p)$ random variables. Since $|W_i \cap S|\leq r$ for all $i$, we have $\sum_i |W_i \cap S|^2\leq r|S|$ and so the Azuma-Hoeffding inequality gives 
\begin{align}\label{eq:Azuma}
\P(|N_{G_B'}(v)\cap S|\geq p|S| +t) \leq \exp \big(-2t^2/(r|S|) \big)\, .
\end{align}
Applying \eqref{eq:Azuma} with a fixed realisation $S=N_{G'_R}(v)$ with $|S|=(1+o(1))np$ and recalling that $r\leq \sqrt{np}$ shows that  $|N_{G_B'}(v)\cap N_{G_R'}(v)|\leq n^{1-1/(4\ell)}p$ with probability at least $1-1/n^2$. A union bound over $v$ shows that~\eqref{eq:double-degree} holds with high probability.
\end{proof}

\section{The number of \texorpdfstring{$J$}{J} to \texorpdfstring{$J$}{J} walks}

When calculating if a given set $I \in [n]^{(k)}$ is independent, we will make use of the event $\cA$ and first take a set of representatives $J \subset I$, with $|J| \geq \delta k$ so that $|\pi_R(J)| = |J|$ or $|\pi_B(J)| = |J|$. We will then use the following lemma to show there are few pairs in $J^{(2)}$ that close a cycle of length $\ell$. Recall from \eqref{eq:def-delta} that we set $\delta=(\log n)^{-1}$.

\begin{lemma}\label{lem:path-counting}
Let $J \in [n]^{(\delta k)}$. On the event $\cA$, the number of walks 
of length $\ell-1$, in $G'$, starting and ending in $J$ is at most 
\begin{equation}\label{eq:path-counting} \delta^{-2} 2^\ell p^{\ell-1}n^{\ell-2} |J|^2.\end{equation}
\end{lemma}

One might be temped to prove the above lemma by first showing that for a fixed $J$ the above bound holds with extremely high probability and then union bounding over all $J$.
While we believe that this method can be made to work, the existence of vertices of $G'$ with atypically high degree into $J$ (which may exist) rules out any simple application of standard concentration inequalities.

To prove Lemma~\ref{lem:path-counting} we instead use a spectral argument, which relies on the quasi-random properties of the graph $G'$. We stress that the reader should not confuse what we are doing here with the ``spectral method'' of \cite{mubayi2024}. 

For this, we first note that we can find a compact expression for the adjacency matrix of the multi-graph $G'$ (here edges are recorded with multiplicity). Recall that $A_R,A_B$ are the adjacency matrices of $G^{\circ}_R$ and $G^{\circ}_B$. We first note that the adjacency matrices of $G_B'$ and $G_R'$ can be obtained from $A_R,A_B$ by taking a tensor product with the $r\times r$ all ones matrix $J_r$. That is, if $A'_R$, $A'_B$ are the adjacency matrices of $G_R',G'_B$, we have that $A'_R =  A_R \otimes J_{r} $ and $A'_B =  A_B \otimes J_{r}$. We then observe that the adjacency matrix of $G'$ is
\[ A = A_R \otimes J_{r}  +  P^t\big( A_B \otimes J_{r}\big) P ,\]
where $P$ is a permutation matrix of a permutation which maps $V_i$ to $W_i$ for all $i$. The key here is the following observation. 

\begin{observation}\label{obs:J-Jwalk}
The number of walks from $J$ to $J$ in $G'$ is precisely $\langle \1_J , A^{\ell-1} \1_J \rangle$.
\end{observation}

To estimate this inner product, we use the spectral theorem to express
\begin{align}\label{eq:Adecomp} A = \mu \big( v \otimes v \big) + M, 
\end{align}
where $\mu$ is the largest eigenvalue of $A$, $v$ is a corresponding unit eigenvector and $v \in \ker(M)$. It follows that
$A^{\ell-1} = \mu^{\ell-1}( v \otimes v) + M^{\ell-1}$ and thus
\begin{equation}\label{eq:spectral-expansion} \langle \1_J , A^{\ell-1} \1_J \rangle = \mu^{\ell-1}\langle \1_J, v \rangle ^2 + \langle \1_J,  M^{\ell-1}\1_J  \rangle \leq |J|^2\mu^{\ell-1}\|v\|_{\infty} + |J|^2\|M\|_{op}^{\ell}  . \end{equation}
We now understand this quantity in three (fairly easy) steps. First we would like to show that $\mu \approx pn$. We then show that the operator norm of $M$ is small, which implies that the second term on the right hand side of \eqref{eq:spectral-expansion} is small. Finally, and in the only difficult step, we show $\|v\|_{\infty} \approx n^{-1/2}$. These together allow us to conclude the desired bound.  

We first observe the following.

\begin{observation}\label{obs:mu} On the event $\cA$ we have $\mu = (1+o(1))2pn$.
\end{observation}
\begin{proof}
 By considering the maximum row sum we have that $\mu \leq (1+O(n^{-c}))2pn$ where $c$ is as in~\eqref{eq:degrees}. On the other hand, by considering the Rayleigh quotient we have 
\[ \mu \geq n^{-1} \langle \1 ,  A \1 \rangle = (1+O(n^{-c}))2pn. \qedhere \]
\end{proof}

To show $\|M\|_{op}$ is small we make the following observation. 

\begin{observation} On the event $\cA$ we have
\[ \big \| A_R \otimes J_{r} - pJ_n \big\|_{op} \leq 3\sqrt{rpn}. \]
\end{observation}
\begin{proof}
From the property \eqref{eq:spectrum}, on the event $\cA$ we have $\|A_R - p J_{n/r}\|_{op} \leq 3\sqrt{pn/r}$. Now write 
\[ A_R \otimes J_{r} - pJ_n = (A_R - pJ_{n/r}) \otimes J_r \] and recall that the eigenvalues of this tensor product are the pairwise products of the eigenvalues of $A_R - pJ_{n/r}$ and $J_r$.     \end{proof}

\begin{lemma}\label{lem:A-decomp} On the event $\cA$ we have $\|M\|_{op} \leq 6\sqrt{rpn}$.
\end{lemma}
\begin{proof}
Assume $\|M\|_{op} > 6\sqrt{rpn}$. 
Since $\mu = (1+o(1))2pn \gg \sqrt{rpn}$, there exists orthogonal unit eigenvectors $u,v$ with 
$\|Au\|_{2}, \|Av\|_2 > 6\sqrt{rpn}$. Thus there is a unit vector $w \in \Span\{u,v\}$  with $\langle w, \1 \rangle = 0$ and $\|Aw\|_2 > 6\sqrt{rpn}$. But this contradicts that 
\[ \big\| A - 2pJ_{n}\big\|_{op} \leq  \big\| A_R \otimes J_{r} -pJ_n \big\|_{op} +  \big\| P^t\big( A_B \otimes J_{r} - pJ_{n} \big)P \big\|_{op} \leq 6\sqrt{rpn}. \qedhere \]
\end{proof}

We now prove our bound on $\|v\|_{\infty}$.

\begin{lemma}\label{lem:perron-evector-flat} On the event $\cA$, we have  $\| v \|_{\infty} \leq \delta^{-1} n^{-1/2}$.
\end{lemma}
\begin{proof}
By the Perron-Frobenius theorem, we know that $v$ has non-negative entries. Without loss of generality, assume $v_1$ is the largest entry of $v$ and assume for a contradiction that $v_1\geq \delta^{-1}n^{-1/2}$.

Let $\gamma_0>0$ be such that $\gamma_0^{1/2^{\ell+1}}=\delta/100$. We claim that for all $\gamma \geq \gamma_0$ we have \begin{equation}\label{eq:propagate-big-evec}
v_x \geq (1-\gamma)v_1 \implies \big|\big\{y \in N_{G'}(x) : v_y \geq (1-\gamma^{1/2})v_1 \big\}\big| \geq 
(1-2\gamma^{1/2})d_{G'}(x)\,.
\end{equation} To see this we prove the contrapositive. Note that  $A v=\mu v$ 
which implies
\begin{equation}\label{eq:eval-equation}
\mu v_x =  \sum_{y\sim x} v_y + O\big(pn^{1-\eta} v_1 \big),
\end{equation}
where the error term is from the contribution of edges coloured both red and blue and we use \eqref{eq:double-degree} to control this term.
Then by the contrapositive assumption for $x$ we have
\begin{align*}
\mu v_x&\leq (1-2\gamma^{1/2})d_{G'}(x)\cdot v_1+2\gamma^{1/2}d_{G'}(x)\cdot(1-\gamma^{1/2})v_1 + O\big(pn^{1-\eta} v_1 \big)\\ &= (1-2\gamma)d_{G'}(x) v_1 +  O\big(pn^{1-\eta} v_1 \big)\, .
\end{align*}
Note that by~\eqref{eq:degrees} and~\eqref{eq:double-degree} we have  $d_{G'}(x)=(1+o(1))2pn$ and so $\mu=(1+o(1))d_{G'}(x)$ by Lemma~\ref{lem:A-decomp}. After noting that $\gamma\geq \gamma_0 \gg n^{-\eta}$, we conclude that $v_x<(1-\gamma)v_1$. This establishes~\eqref{eq:propagate-big-evec}.

 Now for $t\geq 1$, define 
\[ S_t = \big\{ x : v_x \geq (1-\gamma_0^{1/2^{t}})v_1 \big\} . \]

Let $\tau$ be the least integer such that $|S_\tau| \geq (2p)^{-1}$. We show $\tau \leq \ell$. Now assume that $|S_t|\leq (2p)^{-1}$ for some $t\leq \ell$, otherwise we are already done.
 By Observation~\ref{obs:vxexpand}, we have
 \begin{align}\label{eq:StExp}
 |N_{G'}(S_{t})|\geq \delta np |S_t|/8\, .
  \end{align}
 Moreover, for each $x\in S_t$, the degree of $x$ into $N_{G'}(S_{t})\backslash S_{t+1}$ is at most $2\gamma_0^{1/2^{t+1}}d_{G'}(x)$ by \eqref{eq:propagate-big-evec}. It follows that 
 \[ |N_{G'}(S_{t})\backslash S_{t+1}|\leq (1+o(1))4np \gamma_0^{1/2^{t+1}}  |S_t|\leq \delta np|S_t|/16 ,\] recalling that $t\leq \ell$, and therefore \begin{equation}\label{eq:St-grow}
     |S_{t+1}|\geq \delta np |S_t|/16
 \end{equation} by \eqref{eq:StExp}. By~\eqref{eq:St-grow} and the fact that $np\geq n^{1/(\ell-1)}$ we have that $\tau\leq \ell$. 
 
 Let $S_\tau'\subset S_\tau$ with $|S_\tau'|= (2p)^{-1}$. Then arguing as above we see that $|S_{\tau+1}|\geq \delta np  |S'_{\tau}|/16=\delta n/32$. It follows that 
 \[\|v\|_2^2\geq (\delta n/32)(1-\gamma_0^{1/2^{\tau+1}})v_1^2>1,\] a contradiction. 
\end{proof}

We may now finish the proof of Lemma~\ref{lem:path-counting}.

\begin{proof}[Proof of Lemma~\ref{lem:path-counting}] 
We recall Observation~\ref{obs:J-Jwalk} that the number of $J$ to $J$ walks in $G'$ is
$\langle \1_J , A^{\ell-1} \1_J \rangle$. As before, apply~\eqref{eq:Adecomp} to write  
$A^{\ell-1} = \mu^{\ell-1}( v \otimes v) + M^{\ell-1}$ and thus
\[ \langle \1_J , A^{\ell-1} \1_J \rangle = \mu^{\ell-1}\langle \1_J, v \rangle ^2 + \langle \1_J,  M^{\ell-1}\1_J  \rangle. \]
For the first term, we have $\mu = (1+o(1))2pn$ by Lemma~\ref{lem:A-decomp} and by Lemma~\ref{lem:perron-evector-flat}
\[ \langle \1_J, v \rangle  = \sum_{i \in J} v_j \leq \delta^{-1}|J|n^{-1/2}\, . \]
 For the second, we have by Lemma~\ref{lem:A-decomp}
\[ \langle \1_J,  M^{\ell-1} \1_J \rangle \leq |J| \| M^{\ell-1}\|_{op }\leq |J|\|M\|_{op}^{\ell-1} \leq  |J|(6\sqrt{rpn})^{\ell-1}. \]
Putting this together we have that the number of $J$ to $J$ walks is at most 
\[ \delta^{-2} 2^\ell p^{\ell-1}n^{\ell-2} |J|^2\]
since $r \leq (pn)^{1-\frac{2}{\ell-1}}$ and $|J|=\delta k\geq 1/p$, by our choices at \eqref{eq:p-def} and since $\ell \geq 5$.
\end{proof}

\section{Independent set calculation}
For the independent set calculation, we fix $I\in [n]^{(k)}$ and estimate the probability that $I$ is an independent set in $G'-F$. For this, it will be crucial to ensure that not too many edges in $G'\cap I^{(2)}$ are removed in our edge deletion step.  
 Here our first step is to show that the event $\cA$ guarantees this. Say that a walk in $G' = G'_R \cup G'_B$ is \emph{alternating} if it alternately takes steps in the red and blue graphs $G_R', G_B'$.

 We define the set of \emph{closed} pairs by
\[ \cC = \big\{ xy \in [n]^{(2)} : \text{ there exists a non-alternating } xy \text{-walk of length } \ell-1 
\text{ in } G'\big\}. \]
We highlight that only $e \in \cC$ are in danger of being removed in our edge deletion step.
The key point is that if $xy$ is joined by a non-alternating walk
then 
$xy$ must be joined by $\geq r$ non-alternating {walks of the same length.   
We use this to bound the number of closed edges inside of an independent set.

\begin{lemma}\label{lem:closed-edges} On the event $\cA$, we have  $|J^{(2)} \cap \cC| = o(|J|^{2})$ for every $J \in [n]^{(\delta k)}$.
\end{lemma}
\begin{proof}
We note carefully that  
\[ \big| \cC \cap J^{(2)}\big| \cdot r\, \leq |\{ \text{walks of length $\ell-1$ starting and ending in } J \}| \leq  \delta^{-2} 2^\ell p^{\ell-1}n^{\ell-2} |J|^2,\]
where the walks are counted in the graph $G'$ and the second inequality holds by Lemma~\ref{lem:path-counting} and the fact that $\cA$ holds. Thus we have
\[ \big| \cC \cap J^{(2)}\big| \leq \delta^{-2} 2^\ell \big( p^{\ell-1}n^{\ell-2} r^{-1}\big) |J|^2 \ll \delta^{-2}(\log n)^{-2} |J|^{2} = |J|^2, \]
where we used that our choice of $r,p$ satisfy \eqref{eq:inequalities}.
\end{proof}

We are now ready to prove our main lemma on independent sets. For any (multi) graph $H$, let $\cI(H)$ be the set of independent sets in $H$. For this, we recall that $F$ is the set of edges we deleted from $G'$ to form $G$ edges. In what follows, we partition $F$ based on their colour $F=F_R\cup F_B$.

\begin{lemma}\label{lem:ind-sets} For $I \in [n]^{(k)}$ we have
\[ \PP\big( I \in \cI(G'-F) \wedge \cA  \big) \leq (1-p)^{(\delta k)^2/4}. \]
\end{lemma}
\begin{proof} Fix $I\subset V(G')$ and assume that $\cA$ holds. Without loss, assume that $|\pi_R(I)| \geq \delta k$ by \eqref{item:A2} in the definition of $\cA$. Let $J \subset I$ be such that $|\pi_R(J)| = |J| = \delta k$. We restrict our attention to bounding the probability that $J$ is independent in the red graph. That is
\begin{equation}\label{eq:ind-set-calc1}  
\PP\big( I \in \cI(G'-F) \wedge \cA  \big)  \leq\PP\big( J \in \cI(G_R'-F_R)  \wedge \cA  \big) . 
\end{equation}

We now define graphs $G_0 \subset G_1 \subset \cdots \subset G'$, incrementally exposing more randomness. We first expose all of the pairs in $G_B^\circ$ and all of pairs in $G_R^\circ$ that are not in $\pi_R(J)$. That is, we define the random graph $G_0\subset G'$ by
\[ G_0 = G'_B  \cup \big\{ f\in G'_R : \pi_R(f)\not\subset \pi_R(J)  \big\}\, . \]
For $i\ge 0$, and given $G_i$, define $\cC_i \subset J^{(2)}\cap\cC$ to be the set of $xy$ for which there exists a non-alternating $xy$-walk of length $\ell-1$ in $G_i$. 
Define $G_{i+1}$ to be 
\[ G_{i+1} = \big\{ e \in G_R' : \pi_R(e) = \pi_R(f) \text{ for some } f\in \cC_i \big\} \cup G_i. \]

Now define  
\[ G_{\ast} = \bigcup_{i \geq 0} G_i \qquad \text{ and } \qquad \cO = J^{(2)} \setminus \bigcup_{i\geq 0} \cC_i.\]

We now make the following deterministic claim that allows us to move from working with $G_R'-F_R$ to working with $G_R'$.
\begin{claim}\label{cl:ind-set-imp}
If $J \in \cI(G_R'-F_R)$ then $G'_{R} \cap \cO = \emptyset.$ 
\end{claim}
\begin{proof}[Proof of Claim~\ref{cl:ind-set-imp}]
    We prove the contrapositive. Assume $ G_R' \cap \cO \neq \emptyset$ and let $e \in G_R' \cap \cO  $ be minimal with respect to the red ordering  (recall the definition of the ordering from Section~\ref{sec:edge-deletion}). We will show that $e\notin F_R$, implying that $J$ is not an independent set in $G_R'-F_R$.  
    
    For a contradiction, assume that $e \in F_R$ and let $C \subset G'$ be a non-degenerate cycle of length $\ell$ that accounts for the deletion of $e$. 

    Since $e \in \cO$, we have $C \setminus \{e\} \not\subset G_{\ast}$. So let $f' \in C \setminus \{ e\}$ with $f' \not\in G_{\ast}$. We now carefully observe that there exists $f \in \cO \cap G_R' $ such that $\pi_R(f')=\pi_R(f)$, by our definition of $\cO$. 
    Indeed if $\pi_R(f') \not\subset \pi_R(J)$ then $f' \in G_{\ast}$, a contradiction. Also if 
    \[ \pi_R(f') = \pi_R(g) \hspace{0.3em} \text{ for some  } g \in \cC_{\ast} \qquad \Longrightarrow \qquad f' \in G_{\ast}, \]
    which is also a contradiction. Also, note that $f'$ must be red, otherwise $f' \in G_{\ast}$. Thus, taking these together, implies $f \in \cO \cap G_R'$. 
    
    Now notice that $e \not= f$, since $C$ is non-degenerate.
    Thus,  by minimality of $e \in \cO \cap G_R'$, we know that $e$ precedes $f$ in the red ordering. Then $e$ must also precede $f'$, by the definition of the ordering. 
    But then $e$ is not the largest red edge in $C$, in the red ordering, and so it would not have been deleted. This contradicts that $C$ accounts for the deletion of $e$. Thus $e \notin F_R$ which completes the proof.  
 
\end{proof}

We now can finish the proof of Lemma~\ref{lem:ind-sets}. 
We note importantly, that since $\cA$ holds we can apply Lemma~\ref{lem:closed-edges} to learn that 
\[ |\cO| \geq | J^{(2)}\setminus \cC | \geq |J|^2/4 \geq (\delta k)^2/4 .\]
Thus we may condition
\[ 
\PP\big( J \in \cI(G_R'-F_R)  \wedge \cA \big) \leq \max_{G_t,\cO}\, \PP\big(  J \in \cI(G_R'-F_R)  \, \big|\,  G_{\ast},\cO \big) , 
\]
where the maximum is over all $(G_{\ast},\cO)$ which satisfy $|\cO|\geq (\delta k)^2/4$. This holds on the event $\cA$.
We now use Claim \ref{cl:ind-set-imp} to see the above is 
\[ 
\max_{G_{\ast},\cO}\, \PP\big( \cO \cap G'_R = \emptyset \, \big|\, G_{\ast},\cO  \big) \leq \max_{|\cO| \geq (\delta k)^2/4}(1-p)^{|\cO|} \leq (1-p)^{\delta^2 k^2/4}, 
\]
where the first inequality holds by independence of the edges $\cO \cap G'_R$.
\end{proof}

We may now prove our main technical theorem, Theorem~\ref{thm:main2}.

\begin{proof}[Proof of Theorem~\ref{thm:main2}]
Given our preparations, the proof of our main technical theorem now follows quickly. The graph $G=G'-D-F$ has $(1-o(1))n$ vertices with probability $1-o(1)$, by Lemma~\ref{lem:vx-del-2}. It has no cycles of length $\ell$, by construction. Thus we only need to check the independence number. 

For this, first note that every independent set in $G$ is an independent set in $G'-F$. Using Lemma~\ref{lem:ind-sets}, we see the expected number of independent sets of size $k$ in $G$ is thus at most 
\[ \binom{n}{k}(1-p)^{\delta^2 k^2/4 } \leq \exp\big( k \big( \log n  -  p \delta^2 k/4 ) \big).  \]
So if $p\delta^2 k > 8 \log n$, there is no independent set of size $k$, with probability $1-o(1)$.
\end{proof}

\section*{Acknowledgments}
The authors thank the anonymous referee for a very careful reading of an earlier draft of this paper. Their comments led to significant improvements to the presentation.

\bibliography{bib.bib}
\bibliographystyle{abbrv}

\end{document}